\documentclass[12pt]{article}
\usepackage{amsmath,amsthm,mathrsfs}
\usepackage{amsfonts}
\usepackage{latexsym}
\usepackage{cite}
\usepackage{tikz,color}
\usetikzlibrary{chains}
\usepackage{graphicx}
\usepackage{color}
\usepackage{ifpdf}

\newtheorem{theorem}{Theorem}
\newtheorem{lemma}[theorem]{Lemma}
\newtheorem{corollary}[theorem]{Corollary}
\newtheorem{definition}[theorem]{Definition}
\newtheorem{observation}[theorem]{Observation}

\begin{document}

\title{On the general position numbers of maximal outerplanar graphs
\thanks{ The work is supported by NNSF of China (Grant No. 12271251), Postgraduate Research
\& Practice Innovation Program of Jiangsu Province (No. KYCX22\_0323) and the Interdisciplinary Innovation Fund for Doctoral Students of Nanjing University of Aeronautics and Astronautics (No. KXKCXJJ202204). \newline $\dag$: Corresponding author. Email addresses: jingtian526@126.com (J. Tian), kexxu1221@126.com or xukx1005@nuaa.edu.cn (K. Xu), chaodaikun0715@163.com (D. Chao). } }

\author{Jing Tian$^{a,b}$, Kexiang Xu$^{a,b,\dag}$, Daikun Chao$^c$}

\date{}

\maketitle
\vspace{-0.8 cm}
\begin{center}
{ \small  $^a$ College of Mathematics, Nanjing University of
 Aeronautics \& Astronautics,\\
 Nanjing  210016, China}\\
{ \small  $^b$ MIIT Key Laboratory of Mathematical Modelling and High Performance\\
 Computing of Air Vehicles, Nanjing 210016, China}\\
{\small $^c$ College of Automation Engineering,
 Nanjing University of Aeronautics\\ \& Astronautics, Nanjing 210016, China}

\end{center}

\begin{abstract}
A subset $R\subseteq V(G)$ of a graph $G$ is a general position set if any triple set $R_0$ of $R$ is non-geodesic in $G$, that is, no vertex of $R_0$ lies on any geodesic between the other two vertices of $R_0$ in $G$.
Let $\mathcal{R}$ be the set of general position sets of a graph $G$.
The general position number of a graph $G$, denoted by $gp(G)$, is defined as $gp(G)=\max\{|R|:R\in\mathcal{R}\}$.
In this paper, we determine the  bounds on the gp-numbers for any maximal outerplane graph and characterize the corresponding extremal graphs.
\end{abstract}

\noindent {\bf Keywords:} general position set; general position number; maximal outerplane graph

\medskip\noindent
{\bf AMS Math.\ Subj.\ Class.\ (2020)}: 05C12, 05C05, 05C76

\section{Introduction}
\label{S:intro}

All finite connected graphs considered in this paper are simple and undirected.
Let $G$ be a graph with vertex set $V(G)$ and edge set $E(G)$. For simplicity, set $n=|V(G)|$, the order of a graph $G$.
For a vertex $v\in V(G)$, let $N_G(v)$ and $N_G[v]$ denote the open neighborhood and the closed neighborhood of $v$, respectively; thus
$N_G(v)=\{u:uv\in E(G)\}$ and $N_G[v]=\{v\}\cup N_G(v)$. We denote by $d_G(v)=|N_G(v)|$ the degree of $v$. A vertex with degree $d$ is said to be
a \textit{$d$-vertex} in $G$. By $\Delta(G)$ we denote the maximal degree of $G$.
The \textit{distance} $d_G(x,y)$, as usual, is the number of edges on a shortest path between $x$ and $y$ in $G$.
A shortest $x,y$-path is called an $x,y$-\textit{geodesic} in $G$.
For any positive integer $i\geq 2$, let $P_i$ be a path on $p_1$, $p_2$, \ldots, $p_i$ with natural adjacencies.
As usual, $K_n$ denotes the complete graph on $n$ vertices.

A graph is \textit{planar} if it can be embedded in the plane so that no two edges intersect geometrically except at a vertex to which they are both incident.
A \textit{plane graph} is a planar graph with a fixed embedding in the plane.
A plane graph divides the plane into connected regions called \textit{faces}.
The unbounded region is called the \textit{outer face} and each bounded region is called an \textit{inner face}.
A plane graph $G$ is \textit{outerplanar} if it has an embedding in the plane such that all vertices belong to the boundary of its outer face (the unbounded face).
An outerplanar graph $G$ is \textit{maximal} if $G+uv$ is not outerplanar for any two non-adjacent vertices $u$ and $v$ of $G$ and such graph is called a \textit{maximal outerplanar graph} (or just an \textit{MOP} in \cite{Deng:2012}).
A maximal outerplanar graph embedded in the plane is called a \textit{maximal outerplane graph}.

We note that any MOP has a unique Hamiltonian cycle \cite{Rourke:1987} and any MOP is a triangulation graph (or a triangulated disc), that is, a plane graph such that all its faces, except the outer face, are bounded by $K_3$.
Let $f$ be an inner face of a maximal outerplane graph $G$. Then $f$ is isomorphic to a $K_3$.
If $f$ is adjacent to the outer face, i.e., there exists at least one common edge between $f$ and outer face, then we say that $f$ is a \textit{marginal triangle};
otherwise we say that $f$ is an \textit{internal triangle}.
An MOP $G$ without internal triangles is called \textit{striped}.

Since the definition of outerplanar graph was proposed, it has been intensively studied \cite{Bandelt:1986,chvatal-1975,Holst:2007,Honjo:2010,lemanska:2019,Matheson:1996,Tokunaga:2019}.
Motivated with the application of maximal outerlpanar graph, we consider the general position problem for maximal outerlpanar graph in this paper.

The general position problem in graphs was recently introduced by Manuel and Klav\v{z}ar \cite{Manuel:2018} and is now well studied in graph
theory (see, for example, \cite{Hua,Klavzar:2021,klavzar:2021,S.Klavzar-2021,manuel,Neethu:2021,patkos:2020,Tian} for recent papers on this topic).
A subset $R\subseteq V(G)$ of a graph $G$ is a \textit{general position set} if any triple set $R_0$ of $R$ is \textit{non-geodesic} in $G$, that is, no vertex of $R_0$ lies on any geodesic between the other two vertices of $R_0$ in $G$.
Let $\mathcal{R}$ be the set of general position sets of a graph $G$.
The \textit{general position number} of a graph $G$, denoted by $gp(G)$, is defined as $gp(G)=\max\{|R|:R\in\mathcal{R}\}$.
The general position number will be denoted by gp-number briefly.
The general position set of order $gp(G)$ will be called \textit{gp-set} of a graph $G$.

By now some results on the gp-number are obtained for various classes of graph operation.
Several general bounds on the gp-number were presented \cite{Manuel:2018}.
The gp-numbers were determined in \cite{manuel:2018} for a large class of subgraphs of the infinite grid graph and the infinite diagonal grid graph.
And in \cite{Anand:2019}, a characterization of general position set was given and the gp-numbers of bipartite graph and its complement graph were proved.
Recently the gp-numbers in Cartesian products of graphs have been further investigated in \cite{Ghorbani:2021,klavzar-2021,Tian:2021,tian:2021}.
A sharp lower bound on the gp-number for the Cartesian products of graphs was determined and the gp-numbers for joins of graphs, coronas over
graphs, and line graphs of complete graphs were also characterized \cite{Ghorbani:2021}.
The gp-number of the Cartesian product of $n$-dimensional grid graph was proved in \cite{klavzar-2021}.
In particular, it was proved in \cite{tian:2021} that the upper bound on gp-number for Cartesian products of two graphs is sharp and the equality holds if and only if two graphs are both generalized complete graphs.
Moreover, the authors showed that  the gp-number is additive on Cartesian products of trees \cite{Tian:2021}.
In addition, the gp-numbers of other product graphs were investigated and connected with strong resolving graphs \cite{Klavzar-2021}.

In this paper we consider the gp-number of maximal outerplane graphs. Some definitions and basic results are given in Section 2.
In subsequent Section 3, we determine the bounds on the gp-numbers for any  plane graph and characterize its corresponding extremal graphs.
In Section 4, we determine an upper bound on the gp-number for any maximal outerplane graph, and also give some corresponding extremal graphs when the upper bound is achieved.
In addition, we also prove that the bounds on the gp-number  for a maximal outerplane graph containing internal triangles or not, respectively.

\section{Preliminaries}\label{pre}
\label{S:prel}

In this section, we define some concepts and introduce the notations, as well as some results needed later.

For a subset $S\subseteq V(G)$, let $G[S]$ be the induced subgraph of a graph $G$ by $S$.
Let $H_G$ (for short $H_G=H$) be a Hamiltonian cycle of a maximal outerplane graph $G$. Suppose $u$ and $v$ are two vertices on $H$. We call the path on $H$ that goes in the clockwise direction from $u$ to $v$ the \textit{$u,v$-segment} of $H$, and other path the \textit{$v,u$-segment} of $H$, denoted by $S_{uv}$ and $S_{vu}$, respectively.
The \textit{interval} $I_{G}(u,v)$ between vertices $u$ and $v$ is a vertex subset which consists of all vertices lying on some $u,v$-geodesic of  $G$, that is,
$I_{G}(u,v)=\{w:d_G(u,v)=d_G(u,w)+d_G(w,v),w\in V(G)\}$.

A graph $F$ is said to be a \textit{minor} of a graph $G$ if $F$ can be obtained from $G$ by a series of vertex deletions, edge deletions, and edge contractions.
Given a graph $F$, a graph $G$ is said to be \textit{$F$-minor-free} if no minor of $G$ is isomorphic to $F$.
A graph $G$ is outerplanar if and only if it does not contain $K_4$ and $K_{2,3}$ as minors (\cite{syslo:1979}), thus any MOP is $K_4$-minor-free and $K_{2,3}$-minor-free.

 Next, we will give the definition of fan which is important for characterizing of extremal graphs of our main results.

\begin{definition}
A connected graph $G$ of order at least 3 is a fan if it has
only one $(n-1)$-vertex, only two 2-vertices, other vertices (if exist) have degree 3.
\end{definition}

For any positive integer $n\geq 2$, it is observed that a \textit{fan} is $\{v\}\oplus P_{n-1} $ (the join of an isolated vertex $v$ and $P_{n-1}$), denoted by $F_{n-1}$,
 and here the vertex $v$ is called \textit{central vertex} of fan.
A \textit{maximal fan subgraph} of an MOP is a fan which cannot be enlarged by adding a vertex.

In this paper, we always set $[n]=\{1,2,\ldots,n\}$ for a positive integer $n$.
For notations and terminologies not defined here, see \cite{Bondy:1976}. With the above concepts and notations we can recall the following known results.

\begin{lemma}(\cite{Campos:2013})\label{In-k}
Let $G$ be a maximal outerplane graph of order $n\geq4$. If $G$ has $k$ internal triangles, then $G$ has $k+2$ 2-vertices.
\end{lemma}

\begin{lemma}(\cite{Campos:2013})\label{Diam}
If $G$ is a maximal outerplane graph of order $n\geq3$, then $G$ has $n-1$ faces and the Hamiltonian cycle has $n-3$ chords.
\end{lemma}

\begin{lemma}(\cite{Campos:2013})\label{C-M}
Let $v$ be any vertex of a maximal outerplane graph $G$. Then $N_G[v]$ forms a maximal fan in $G$ with $v$ as the central vertex.
\end{lemma}

\begin{lemma}(\cite{Anand:2019})\label{GP-C}
Let $G$ be a connected graph. Then a vertex subset $S$ is a general position set if and only if the components of $G[S]$ are
complete subgraphs, the vertices of which form an in-transitive, distance-constant partition of $S$.
\end{lemma}

\begin{lemma}\label{K-G}(\cite{Chandran:2016})
Let $G$ be a connected graph of order $n\geq 4$. Then
\begin{itemize}
\item [$(i)$] $gp(G)=n$ if and only if $G\cong K_n$;
\item[$(ii)$] $gp(G)=n-1$ if and only if $G\cong K_{n}^{k}$ with $1\leq k\leq n-3$ or $G$ is a non-trivial generalized complete graph.
\end{itemize}
\end{lemma}

It is easily seen that $gp(G)\geq 3$ for any maximal outerplane graph $G$. Note that $G$ is 2-connected since a maximal outerplane graph has Hamiltonian cycle. Then we have $\Delta(G)\geq 2$ for any maximal outerplane graph $G$.
Next, we will give a general lower bound on the gp-number for any maximal outerplane graph $G$ with respect to $\Delta(G)$.

\begin{lemma}\label{D-M}
Let $G$ be a maximal outerplane graph with $\Delta(G)\geq 2$. Then $gp(G)\geq \lfloor\frac{2(\Delta(G)+1)}{3}\rfloor$.
\end{lemma}
\begin{proof}
Assume, without loss of generality, that  $d_G(u)=\Delta(G)$ with $u\in V(G)$.
 It is easy to see that the result holds if  $\Delta(G)\leq 4$. So we may let $\Delta(G)\geq 5$ in the following.
Let $N_G(u)=\{v_1,v_2,\ldots,v_{k-1}\}$ with $k-1=\Delta(G)$.
By Lemma \ref{C-M}, we know that $N_G[u]$ forms a maximal fan $G_u$ with the central vertex $u$ in $G$.

Assume that $G_u=\{u\}\oplus P_{k-1}$ and $V(P_{k-1})=\{v_1,v_2,\ldots,v_{k-1}\}$ with natural adjacencies.
Set $S_1=\{v_{3i-2},v_{3i-1}:i\in[j]\}$ if $k\in\{3j,3j+1\}$ and $S_2=\{v_{3i-2},v_{3i-1}:i\in[j-1]\}\cup \{v_{k-1}\}$ if $k=3j+2$.
Then $|S_1|=2j$ and $|S_2|=2j+1$ since $k\geq 6$.
If $k=3j$ or $k=3j+1$, we have $d_G(w_1,w_2)\in\{1,2\}$ for any two vertices $\{w_1,w_2\}\subseteq S_1$.
Furthermore, if $d_G(w_1,w_2)=1$, then $w_1$ and $w_2$ are two adjacent vertices on $P_{k-1}$.
While $d_G(w_1,w_2)=2$, it means that $d_{P_{k-1}}(w_1,w_2)\geq 2$.
Note that $gp(G)\geq 3$ for any maximal outerplane graph $G$.
But if there exists one vertex $z\in S_1\setminus\{w_1,w_2\}$ such that $d_G(w_1,w_2)=d_G(w_1,z)+d_G(z,w_2)$, then $d_G(w_1,w_2)\geq 3$. It is a contradiction.
Therefore, $S_1$ is a general position set of $G$ and $gp(G)\geq 2j=\frac{2k}{3}$.
Analogously, $S_2$ is also a general position set of $G$, we have $gp(G)\geq 2j+1=\lfloor\frac{2k}{3}\rfloor$, as desired.
\end{proof}

\begin{lemma}\label{N-x}
Let $S$ be a general position set of maximal outerplane graph $G$ with $|S|\geq 3$. If $x\in S$, then $|N_G(x)\cap S|\leq 2$.
\end{lemma}
\begin{proof}
 Note that $d_G(x)\geq 2$ for any $x\in V(G)$.  So we may suppose $d_G(x)\geq 3$, as otherwise the result holds obviously.
Assume further, without loss of generality, that $|N_G(x)\cap S|\geq 3$ and $N_G(x)\cap S=\{v_1,v_2,\ldots, v_k\}$ with $k\geq 3$.
There must be two vertices, say $v_i$ and $v_j$, such that $d_G(v_i,v_j)=2$ with $\{i,j\}\subseteq [k]$, otherwise the induced subgraph on $N_G(x)$ has a $K_4$ minor.
It is impossible. Then $d_G(v_i,v_j)=2=d_G(v_i,x)+d_G(x,v_j)$, which implies that $x\in I_G(v_i,v_j)$ in $G$. It contradicts with $\{v_i,v_j,x\}\subseteq S$.
 Thus, $|N_G(x)\cap S|\leq 2$ for any vertex $x\in V(G)$.
\end{proof}

 By Lemma \ref{N-x},  it is also immediate that the following holds.

\begin{corollary}\label{N-G-U}
Let $S$ be a general position set of any maximal outerplane graph $G$. For any vertex $x\in V(G)$, if $|N_G(x)\cap S|\geq 3$, then $x\not\in S$.
\end{corollary}

Next, we will characterize the induced subgraphs on the general position set for a maximal outerplane graph $G$.

\begin{lemma}\label{GP-S}
Let $S$ be a general position set of any maximal outerplane graph $G$.
If $|S|\geq 4$, then $G[S]$ contains no triangle.
\end{lemma}
\begin{proof}
Let $H$  be a Hamiltonian cycle of a maximal outerplane graph $G$ and $S_{uv}$ be an $u$,$v$-segment on $H$ for any $u$, $v\in V(H)$.
Assume that there exists one triangle $K_3$ in $G[S]$ and $V(K_3)=\{x,y,z\}$ with clockwise order on $H$.

Suppose first that $d_H(x,y)=d_H(y,z)=1$. Then $d_H(x,z)=2$.
As $\{x,y,z\}\subseteq V(K_3)$, our assumption implies that $d_G(y)=2$ and $d_G(x,z)=1$.
Since $|S|\geq 4$, there must be another vertex $w\in S\setminus\{x,y,z\}$.
Applying Lemma \ref{GP-C}, we have $d_G(w,x)=d_G(w,y)=d_G(w,z)$.
It follows that either $d_G(w,y)=d_G(w,x)+1$ or $d_G(w,y)=d_G(w,z)+1$, that is, either $x\in I_G(w,y)$ or $z\in I_G(w,y)$, which are impossible.

Now assume, without loss of generality, that $d_H(x,y)=1$ and $2\leq d_H(x,z)\leq \lfloor\frac{n}{2}\rfloor$.
There also must be another $w\in S\setminus\{x,y,z\}$ since $|S|\geq 4$.
It follows from Lemma \ref{GP-C} that $d_G(w,x)=d_G(w,y)=d_G(w,z)$.
If $w\in V(S_{yz})$, it is obvious that either $d_G(w,x)=d_G(w,y)+1$ or $d_G(w,x)=d_G(w,z)+1$, that is, either $y\in I_G(w,x)$ or $z\in I_G(w,x)$. It is impossible.
Analogously, we can get the same contradiction if $w\in V(S_{zx})$ or $\min\{d_H(x,y),d_H(x,z),d_H(y,z)\}\geq 2$. This completes the proof.
\end{proof}

\section{Maximal outerplane graphs}
\label{S:maxi}

In this section, we first give the formula of the gp-number for a fan.
 Then we will determine the upper bound on the gp-number of a maximal outerplane graph $G$ and characterize some corresponding extremal graphs when the bound is achieved.

For convenience, $G$ denotes a maximal outerplane graph in this section unless otherwise specified.

\begin{observation}\label{NV}
Let $H$ be a Hamiltonian cycle of $G$ and
 $S_{uv}$ be an $u,v$-segment on $H$. If $d_G(u,v)=1$, then $N_G(u^{\prime})\subseteq V(S_{uv})$,
 for $u^{\prime}\in V(S_{uv})\setminus\{u,v\}$.
\end{observation}

By symmetry, based on the above Observation \ref{NV}, we also have $N_G(v^{\prime})\subseteq V(S_{vu})$, where $v^{\prime}\in V(S_{vu})\setminus\{u,v\}$.

Next, we will prove that there is one common neighbor vertex for two adjacent vertices on $H$ as follows.

\begin{lemma}\label{N-T}
Let $H$ be a Hamiltonian cycle of $G$.
For two adjacent vertices $u$ and $v$ of $H$, then $u,v\in N_G(h)$ for some $h\in V(H)\setminus\{u,v\}$
\end{lemma}
\begin{proof}
Let $V(H)=\{h_1,h_2,\ldots,h_n\}$ with clockwise order on $H$. The proof is simple if $n\leq 6$.
So we may assume that $n\geq 7$ in what follows.

Without loss of generality, let $u=h_1$ and $v=h_2$.
If $d_G(h_1)=2$ (the case $d_G(h_2)=2$ is symmetry), then we have $d_G(h_2,h_n)=1$. Obviously, $h_1,h_2\in N_G(h_n)$.
 Next, we may consider the case $\min\{d_G(h_1),d_G(h_2)\}\geq 3$.
Using Lemma \ref{C-M}, $N_G[h_1]$ forms a fan with the central vertex $h_1$ in $G$.
Since $h_2\in N_G(h_1)$ and $d_H(h_1,h_2)=1$, $\{h_1,h_2\}\subseteq N_G(h)$ for some $h\in N_G(h_1)\setminus\{h_2\}$, completing this proof.
\end{proof}

Next, we will give the formula of the gp-number of fan which plays a crucial role in the proof of our main result.

\begin{theorem}\label{F-G}
Let $G$ be a fan graph of order $n\geq 5$. Then $gp(G)=\lfloor\frac{2n}{3}\rfloor$.
\end{theorem}
\begin{proof}
Let $G\cong \{v\}\oplus P_{n-1}$ with $V(P_{n-1})=\{1,2,\ldots,n-1\}$.
Assume further that $R$ be a gp-set of $G$. Suppose $n\geq 8$, as otherwise the proof is simple.
Note that $G$ is a fan and $\Delta(G)\geq 7$. It follows from  Lemma \ref{D-M} that $|R|\geq 5$.
Applying Corollary \ref{N-G-U}, it is obvious that $v\not\in R$.
For any positive integer $k\geq 2$,
based on the value of $n$, we consider the following cases.

Suppose first that $n=3k$.
Since $G$ is a fan and $\Delta(G)=n-1$, by Lemma \ref{D-M}, it is clear to see that $gp(G)\geq 2k$.
Let $V_3(i)$ be the set of consecutive three vertices of $P_{n-1}$ and $V_3(i)=\{3i-2, 3i-1, 3i\}$, where $\in[k-1]$.
Then $V_3(i)\cap V_3(i+1)=\emptyset$.
 Clearly, we have $|V_3(i)\cap R|\leq 2$ for any $i\in[k-1]$, otherwise the induced subgraph on $V_3(i)$ is a subpath $P_3$ in $G$ contradicting the fact that $V_3(i)$ is a general position set of $G$.
Then we have
$$\sum\limits_{i=1}^{k-1}|V_3(i)\cap R|\leq 2(k-1).$$
Since $\bigcup\limits_{i=1}^{k-1}V_3(i)=[n-3]=V(P_{n-1})\setminus\{n-2,n-1\}$,
$$|V(G)\cap R|\leq 2(k-1)+2=2k.$$
Hence $gp(G)=2k$ if $n=3k$.

Analogously, we can get
$$gp(G)=\left\{
\begin{array}{rcl}
2k, &{ n=3k+1}\\
2k+1, &{n=3k+2}.
\end{array}\right.
$$
In consequence, we can conclude that $gp(G)=\lfloor\frac{2n}{3}\rfloor$, completing this proof.
\end{proof}

We now construct two classes of connected graphs, which we will use in the main result of proofs.
For the sake of clarity, we give the definitions as follows.

\begin{definition}
For any $n\geq 6$, let $F_{n-2}=\{v\}\oplus P_{n-2}$ be a fan graph with $V(P_{n-2})=\{p_1,p_2,\ldots,p_{n-2}\}$.
Then the quasi-fan graph $F(i;n)$ is a connected graph obtained from $F_{n-2}$ and an isolated vertex $u$ by adding two edges $up_{i}$ and $up_{i+1}$, where $i\in [n-3]$.
\end{definition}

\begin{definition}
For any two positive integers $n\geq 6$ and $1\leq t\leq \lfloor\frac{n}{3}\rfloor-1$, let $F_{3t}=\{v\}\oplus P_{3t}$ and $F_{n-3t-1}=\{u\}\oplus P_{n-3t-1}$ be two fans with $V(P_{3t})=\{p_1,p_2,\ldots,p_{3t}\}$ and $V(P_{n-3t-1})=\{u_1,u_2,\ldots,u_{n-3t-1}\}$.
If $u=p_{3j-1}$, that is, gluing $u$ to $p_{3j-1}$, we denote this connected graph by $DF^{j}(t;n)$.
Moreover,  for $j\in [t]$,
 \begin{itemize}
\item[1.] $G_1(j)$ is obtained by adding new edge between $p_{3j-2}$ and $u_1$ in $DF^{j}(t;n)$;
\item[2.] $G_2(j)$ is obtained by adding new edge between $p_{3j}$ and $u_{n-3t-1}$ in $DF^{j}(t;n)$.
\end{itemize}
\end{definition}

\begin{figure}[htbp]
\centering
\begin{tikzpicture}[scale=1.5]

	\node[fill=black,circle,inner sep=2pt] at (0,0) {};
	\node[fill=black,circle,inner sep=2pt] at (1,0) {};
    \node[fill=black,circle,inner sep=2pt] at (2,0) {};
    \node[fill=black,circle,inner sep=2pt] at (-1,0) {};
    \node[fill=black,circle,inner sep=2pt] at (-2,0) {};
    \node[fill=black,circle,inner sep=2pt] at (-0.5,-1) {};
    \node[fill=black,circle,inner sep=2pt] at (0.5,-1) {};
    \node[fill=black,circle,inner sep=2pt] at (-0.5,1) {};

    \node [above=0.5mm ] at (0,0) {$u$};
    \node [below=0.5mm] at (-2,0) {$p_1$};
    \node [above=0.5mm] at (-1,0) {$p_{3j-2}$};
    \node [below=0.5mm] at (1,0) {$p_{3j}$};
    \node [below=0.5mm] at (2,0) {$p_{3t}$};
    \node [below=0.5mm] at (-0.5,-1) {$u_1$};
    \node [below=0.5mm] at (0.5,-1) {$u_{n-3t-1}$};
    \node [above=0.5mm] at (-0.5,1) {$v$};

	\draw [thick] (-0.5,1)--(-2,0);
    \draw [thick] (-0.5,1)--(-1,0);
    \draw [thick] (-0.5,1)--(0,0);
    \draw [thick] (-0.5,1)--(1,0);
    \draw [thick] (-0.5,1)--(2,0);
    \draw [thick][dashed] (-2,0)--(-1,0);
    \draw [thick] (-1,0)--(0,0);
    \draw [thick] (0,0)--(1,0);
    \draw [thick] [dashed](1,0)--(2,0);
    \draw [thick] (0,0)--(-0.5,-1);
    \draw [thick] (0,0)--(0.5,-1);
    \draw [thick] [dashed](-0.5,-1)--(0.5,-1);
    \draw [thick,blue] (-1,0)--(-0.5,-1);
\end{tikzpicture}
\hspace{10mm}
\begin{tikzpicture}[scale=1.5]

	\node[fill=black,circle,inner sep=2pt] at (0,0) {};
	\node[fill=black,circle,inner sep=2pt] at (1,0) {};
    \node[fill=black,circle,inner sep=2pt] at (2,0) {};
    \node[fill=black,circle,inner sep=2pt] at (-1,0) {};
    \node[fill=black,circle,inner sep=2pt] at (-2,0) {};
    \node[fill=black,circle,inner sep=2pt] at (-0.5,-1) {};
    \node[fill=black,circle,inner sep=2pt] at (0.5,-1) {};
    \node[fill=black,circle,inner sep=2pt] at (-0.5,1) {};

    \node [above=0.5mm ] at (0,0) {$u$};
    \node [below=0.5mm] at (-2,0) {$p_1$};
    \node [below=0.5mm] at (-1,0) {$p_{3j-2}$};
    \node [above=0.5mm] at (1,0) {$p_{3j}$};
    \node [below=0.5mm] at (2,0) {$p_{3t}$};
    \node [below=0.5mm] at (-0.5,-1) {$u_1$};
    \node [below=0.5mm] at (0.5,-1) {$u_{n-3t-1}$};
    \node [above=0.5mm] at (-0.5,1) {$v$};

	\draw [thick] (-0.5,1)--(-2,0);
    \draw [thick] (-0.5,1)--(-1,0);
    \draw [thick] (-0.5,1)--(0,0);
    \draw [thick] (-0.5,1)--(1,0);
    \draw [thick] (-0.5,1)--(2,0);
    \draw [thick][dashed] (-2,0)--(-1,0);
    \draw [thick] (-1,0)--(0,0);
    \draw [thick] (0,0)--(1,0);
    \draw [thick] [dashed](1,0)--(2,0);
    \draw [thick] (0,0)--(-0.5,-1);
    \draw [thick] (0,0)--(0.5,-1);
    \draw [thick] [dashed](-0.5,-1)--(0.5,-1);
    \draw [thick,blue] (1,0)--(0.5,-1);
\end{tikzpicture}

\caption{ Graphs $G_1(j)$ (left) and $G_2(j)$ (right). }\label{figure1}
\end{figure}
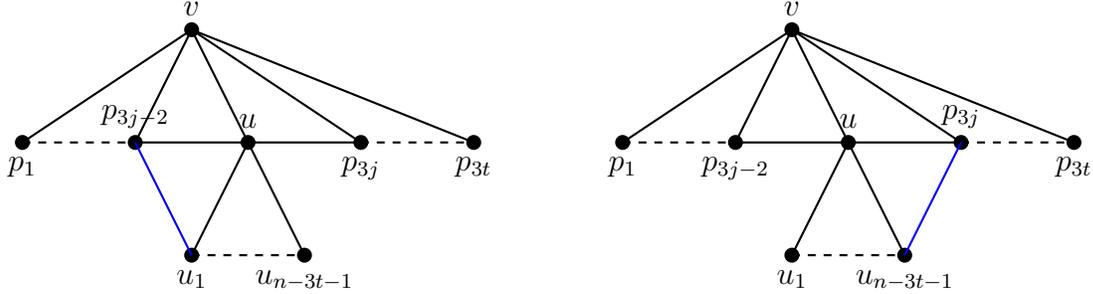

These graphs $G_1(j)$ and $G_2(j)$ are shown in Figure 1.

It is observed that $G_1(1)\cong F_{n-1}\cong G_2(1)$ if $t=1$.
For convenience, we may assume that $\mathcal{F}=\{F(i;n):1\leq i\leq n-2\}$ and $\mathcal{G^*}=\{G_k(j):1\leq k\leq 2\}$.

Lemma \ref{D-M} gives a general lower bound on gp-number for any maximal outerplane graph. Next, we will give the proof of our main result.

 \begin{theorem}\label{F^*-I-N}
Let $G$ be a maximal outerplane graph of order $n\geq 6$. Then $gp(G)\leq \lfloor\frac{2n}{3}\rfloor$.
Moreover, if $n\not\equiv1(mod~3)$, the equality holds if and only if $G\cong F_{n-1}$; otherwise, the equality holds if and only if $G\in\mathcal{F}\cup \mathcal{G}^*$.
\end{theorem}
\begin{proof}
Let $H$ be the Hamiltonian cycle of $G$ and $h_1$,$h_2$,\ldots,$h_n$ be its vertices in a cycle clockwise order.
Assume further that $R$ is a gp-set of $G$ and $V_{3}$ is the set of consecutive three vertices on $H$.
Recall that $gp(G)\geq 3$ for any maximal outerplane graph.
The upper bound holds obviously if $gp(G)\in\{3,4\}$ since $n\geq 6$.
So we may assume that $|R|\geq 5$.
In order to obtain the upper bound, the following claim is important for our proof.

\medskip\noindent
{\bf Claim 1.} $|V_3\cap R|\leq 2$ for any $V_3\subseteq V(H)$.\\
Suppose that $V_3\subseteq R$, and without loss of generality let $V_3=\{u,v,w\}$ with clockwise order on $H$.
Then $d_H(x,z)=2$ and $d_H(x,y)=d_H(y,z)=1$.
Note that $d_G(x,z)\leq d_H(x,z)$.
The induced subgraph on $V_3$ is a $K_3$ if $d_G(x,z)=1$, which contradicts Lemma \ref{GP-S}.
But if $d_G(x,z)=2$, we have $d_G(x,z)=d_G(x,y)+d_G(y,z)$. It means that $y\in I_G(x,z)$, which is impossible.
Thus, $|V_3\cap R|\leq 2$.
This completes the claim.

Let $V_3(i)=\{h_{3i-2},h_{3i-1},h_{3i}\}$ be the set of consecutive three vertices of $H$, where $i\in[k]$.
Assume first that $n=3k$. Note that $k\geq 2$ is a positive integer as $n\geq6$. It implies that $\bigcup_{i=1}^{k}V_3(i)=V(H)$.
According to Claim 1, it is clear to verify that $\sum_{i=1}^{k}|V_3(i)\cap R|\leq 2k=\lfloor\frac{2n}{3}\rfloor$.

Now, suppose that $n=3k+1$. Then $\bigcup_{i=1}^{k}V_3(i)=V(H)\setminus\{h_n\}$.
If $|R\cap V(H)|\geq 2k+1$, applying Claim 1, then it implies that $h_n\in R$ and $|V_3(i)\cap R|=2$ for any $i\in[k]$.
Applying Claim 1 again, we can get $|\{h_1,h_2\}\cap R|=1$ and $h_{n-1}\not\in R$.
Furthermore, if $h_1\in R$, then $h_3\in R$. Since $|V_3(i)\cap R|=2$, in view of Claim 1, we may assume that $V_3(i)\cap R=\{h_{3i-2},h_{3i}\}$ for any $i\in[k]$. Then we have $h_{n-1}\in R$ if $i=k$, which is a contradiction.
While $h_2\in R$, it implies that $h_3\in R$ and $V_3(i)\cap R=\{h_{3i-1},h_{3i}\}$ for any $i\in[k]$.
It leads to $h_{n-1}\in R$ if $i=k$, a contradiction again.
Thus, we have $|R|\leq 2k=\lfloor\frac{2n}{3}\rfloor$.
For the case $n=3k+2$, we can prove that $|R|\leq 2k+1=\lfloor\frac{2n}{3}\rfloor$ as similar to the above.
Consequently, we get $gp(G)\leq \lfloor\frac{2n}{3}\rfloor$.

Next, we will prove that the necessary and sufficient condition when the equality holds. Let $k\geq 2$ be a positive integer.
Assume first that $n=3k$ or $n=3k+2$. By Theorem \ref{F-G}, we have $gp(G)=\lfloor\frac{2n}{3}\rfloor$ if $G\cong F_{n-1}$.
Suppose now that $n=3k+1$. If $G\cong F(i;n)$, we may assume without loss of generality that $F_{n-1}=\{v\}\oplus P_{n-1}$ and $V(P_{n-1})=\{p_1,p_2,\ldots,p_{n-1}\}$.
Take $S_1=\{p_{3\ell-2},p_{3\ell-1}:\ell\in[k]\}$. It is easy to prove that $S_1$ of order $2k$ is a general position set of $G$.
By the above argument, we can get $gp(G)=2k=\lfloor\frac{2n}{3}\rfloor$.
 Similarly, we also have $gp(G)=2k$ if $G$ is isomorphic to $G_1(j)$ or $G_2(j)$.

So we prove the sufficiency only in what follows.
Let $G$ be a maximal outerplane graph, which satisfies $gp(G)=\lfloor\frac{2n}{3}\rfloor$.
Suppose $n\geq 12$, as otherwise the proof is simple. Note that $|R|\geq 8$.
Indeed, $|V_3(i)\cap R|\geq 1$ for any $i\in[k]$.
If there exists one $V_3(i)$ such that $|V_3(i)\cap R|=0$, by Claim 1,
we can get
$$|R|\leq 2\times(\lfloor\frac{n}{3}\rfloor-1)+1=2\times\lfloor\frac{n}{3}\rfloor-1<\lfloor\frac{2n}{3}\rfloor.$$
It contradicts our assumption. Combining this with Claim 1, we have $1\leq |V_3(i)\cap R|\leq 2$ for any $i\in[k]$.

Similar to $V_3(i)$,  $V_4$ is defined as the set of consecutive four vertices of $H$.
To prove our result, the following claim will be useful to the rest proof.

\medskip\noindent
{\bf Claim 2.} If $|V_4\cap R|=3$, then $V_4\subseteq N_G[h]$ for some $h\in V(h)$.\\
Let $V_4=\{u,v,w,z\}$ with  clockwise order on $H$.
By Claim 1, we may assume, without loss of generality, that $V_4\cap R=\{u,v,z\}$.
Since $u$, $v$ are two adjacent vertices of $H$, in view of Lemma \ref{N-T}, $\{u,v\}\subseteq N_G(h^{\prime})$ for some $h^{\prime}\in V(H)\setminus\{u,v\}$.
By Lemma \ref{GP-S}, we have $h^{\prime}\neq z$ since $|R|\geq 8$.
It is clear to see that $V_4\subseteq N_G[h^{\prime}]$ if $h^{\prime}=w$.
If $h^{\prime}\in V(H)\setminus V_4$, then $d_G(v,z)=2$. Otherwise $v\in I_G(u,z)$ contradicts $V_4\cap R=\{u,v,z\}$.
As $d_G(v,z)=2$ and $d_G(u,v)=1$, it implies that $d_G(u,z)=2$.
Then there must be a vertex $u^*\in N_G(u)\setminus\{v\}$ or $z^*\in N_G(z)\setminus\{w\}$ such that $d_G(u^*,z)=1$ or $d_G(u,z^*)=1$.
Without loss of generality, assume that there exists a vertex $u^*\in N_G(u)\setminus\{v\}$ such that $d_G(u^*,z)=1$.
By Observation \ref{NV}, $u^*=h^{\prime}$. Since $G$ is a maximal outerplane graph and $d_G(v,z)=2$, we have $d_G(h^{\prime},w)=1$.
Hence $V_4\subseteq N_G(h^{\prime})$.
Moreover, the result also holds if $d_G(u,z^*)=1$ for some $z^*\in N_G(z)\setminus\{w\}$.
This completes the claim.

As $gp(G)=\lfloor\frac{2n}{3}\rfloor$, there must be two vertices $r_1$, $r_2\in R $ such that $d_H(r_1,r_2)=1$.
Otherwise, $|R|\leq \lceil\frac{n}{2}\rceil< \lfloor\frac{2n}{3}\rfloor$ contradicts the fact that $gp(G)=\lfloor\frac{2n}{3}\rfloor$.
Based on the value of $n$, we divide into the following cases.

\medskip\noindent
{\bf Case 1.} $n=3k$.\\
In this case, $gp(G)=2k$ with $k\geq 3$. Since $gp(G)=2k$, by Claim 1, we have $|V_3\cap R|=2$ for any $V_3\subseteq V(H)$.
Otherwise, $gp(G)\leq\frac{2(n-3)}{3}+1=2k-1$, which is impossible.
Without loss of generality, let $h_1$, $h_2\in R $.
Applying Claim 1, $h_3$, $h_n\not\in R$ and $h_4\in R$.
Similar to the above, we can get $h_{3i-2}$, $h_{3i-1}\in R$ and $h_{3i}\not\in R$, for any $i\in [k]$.
Take $S=\{h_{3i-2},h_{3i-1}:i\in[k]\}$ and $\overline{S}=\{h_{3i}:i\in[k]\}$.
It implies that $|S|=2k$, thus $S$ is a gp-set of $G$.
By Lemmas \ref{GP-S} and \ref{N-T}, $h_1$, $h_2\in N_G(h)$ for some $h\in \overline{S}$.
Since $h_1$, $h_2$, $h_3$ and $h_4$ are four consecutive vertices on $H$, in view of Claim 2, $\{h_1,h_2,h_3,h_4\}\subseteq N_G[h^{\prime}]$ for some $h^{\prime}\in V(H)$.
Then $h=h^{\prime}$, otherwise the induced subgraph on $\{h_1,h_2,h,h^{\prime}\}$ has a $K_4$ minor. It is impossible.

Suppose that $h=h_3$.
It implies that $d_G(h_1,h_3)=1$, thus $h_2$ is a 2-vertex in $G$ and $d_G(h_2,h_4)=2$.
Since $d_G(h_4,h_2)=2$ and $\{h_4,h_5\}\subseteq R$, in view of Lemma \ref{GP-C}, $d_G(h_2,h_5)=2$.
By Lemma \ref{GP-S}, we obtain that $d_G(h_3,h_5)=1$, which means that $h_4$ also is a 2-vertex in $G$.
For similar reason as the above, it follows that $h_7\in N_G(h_3)$ and we also have $d_G(h_5,h_7)=2$ since Lemma \ref{GP-S}.
According to the maximal outerplanarity of $G$, it leads to $d_G(h_3,h_6)=1$.
Thus we have $\{h_6,h_7,h_8\}\subseteq N_G(h_3)$.
Analogously,
we can obtain that $h^*\in N_G(h_3)$ for any $h^*\in V(H)\setminus\{h_3\}$.
As a consequence, $G\cong F_{n-1}$.

Assume now that $h\in \overline{S}\setminus\{h_3\}$, and without loss of generality let $h=h_{3i}$ for some $i\in [k]\setminus\{1\}$.
Then it follows that $h_2$ is a 3-vertex in $G$ since $\{h_1,h_2,h_3,h_4\}\subseteq N_G(h)$.
As $h_2$, $h_3$, $h_4$ and $h_5$ are four consecutive vertices on $H$, by using Claim 2, we have $h_5\in N_G(h)$.
Similarly, we can get $h_{3j-2}$, $h_{3j-1}\in N_G(h)$ for any $j\in[k]$ and $h_{3j}\in N_G(h)$ for any $j\in[k]\setminus\{i\}$.
Therefore $G\cong F_{n-1}$.

\medskip\noindent
{\bf Case 2.} $n=3k+1$.\\
In this case, $gp(G)=2k$. Let $R^*$ be the set of isolated vertices in $G[R]$. Then $R^*\subseteq R$.
We will claim that $|R^*|\in\{0,2\}$.

Suppose on the contrary that $|R^*|\geq 3$, and without loss of generality let $h_1$, $h_s$, $h_t\in R^*$ with $3\leq s\leq t\leq \lfloor\frac{n}{2}\rfloor$.
It follows from Claim 1 that
$gp(G)\leq 2\times\big\lfloor\frac{(n-7)}{3}\big\rfloor+3=2k-1$
if $s=3$ and $t=5$, a contradiction.
While $s=3$ and $t\geq 6$,
by Claim 1, we have
\begin{equation*}
\begin{aligned}
gp(G)& \leq 2\times\big\lfloor\frac{(t-6)}{3}\big\rfloor+2+ 2\times\big\lfloor\frac{(n-t-2)}{3}\big\rfloor+3\\
     &=2\times\lfloor\frac{t}{3}\rfloor+ 2\times\big\lfloor{\frac{n-t-2}{3}}\big\rfloor+1\\
     &=2\times\lfloor\frac{t}{3}\rfloor+2\times\big\lfloor{\frac{-1-t}{3}}\big\rfloor+2k+1\\
     &= 2k-1.
\end{aligned}
\end{equation*}
Note that $\lfloor\frac{t}{3}\rfloor+\big\lfloor{\frac{-1-t}{3}}\big\rfloor=-1$.
It is a contradiction.
Using the similar arguments as the above, we can get the same contradiction if $4\leq s\leq t$.
Thus, we can obtain $|R^*|\leq 2$.

Next assume that $|R^*|=1$, and without loss of generality let $h_1\in R^*$.
It means that $\{h_2,h_n\}\cap R=\emptyset$. Let $V^*=V(H)\setminus\{h_1,h_2,h_n\}$.
Then $|V^*\cap R|=2k-1$ since $gp(G)=2k$.
Applying Claim 1 again, there must be another isolated vertex in $G[R]$ as $2k-1$ is odd, which leads to a contradiction.
As a consequence, $|R^*|\in\{0,2\}$, completing this claim.

Based on the order of $R^*$, we divide into following subcases to proof this case.

\medskip\noindent
{\bf Subcase 1.} $|R^*|=0$.\\
In this subcase, for any vertex $u\in V(H)$, if $u\in R$, then we have $|N_H(u)\cap R|=1$ since Claim 1.
So we may assume, without loss of generality, that $R=\{h_{3i-2},h_{3i-1}:i\in[k]\}$ with order $2k$.  Take $\overline{R}=V(H)\setminus R=\{h_{3i}:i\in[k]\}\cup\{h_n\}$.
Since $h_1$, $h_2$, $h_3$, $h_4$ are four consecutive vertices on $H$, by Lemma \ref{GP-S} and Claim 2, $\{h_1,h_2,h_3,h_4\}\subseteq N_G[\bar{r}]$ for some $\bar{r}\in \overline{R}$.

Suppose that $\bar{r}=h_n$ (the case $\bar{r}=h_{n-1}$ is similar).
Then $h_2$ is a 3-vertex of $G$ and $N_G(h_2)=\{h_1,h_3,h_n\}$.
Since $d_G(h_2,h_4)=2$ and $\{h_4,h_5\}\subseteq R$, applying Lemma \ref{GP-C}, we have $d_G(h_2,h_5)=2$, and thus $d_G(h_n,h_5)=1$.
Analogously, $d_G(h_n,h_7)=d_G(h_n,h_8)=1$ and $d_G(h_n,h_6)=1$ since $\{h_7,h_8\}\subseteq R$ and $d_G(h_5,h_7)=2$.
By analogy, we can get $d_G(h_n,h_j)=1$ for any $j\in[n-1]$.
Thus $G\cong F_{n-1}$.

Assume that $\bar{r}\in \overline{R}\setminus\{h_{n-1},h_n\}$ and $\bar{r}=h_{3i}$ with $i\in[k-1]$.
If $i=1$, then $\bar{r}=h_{3}$. It implies that $h_2$ has exactly two neighbors $h_1$ and $h_3$ in $G$ since $d_G(h_1,h_3)=1$.
Since $d_G(h_2,h_4)=2$ and $\{h_4,h_5\}\subseteq R$, it follows from Lemma \ref{GP-C} that $d_G(h_2,h_5)=2$, thus we have $d_G(h_3,h_5)=1$.
Then we get that $h_4$ also is a 2-vertex of $G$.
For the same reason as the above, we have $d_G(h_3,h_7)=d_G(h_3,h_8)=d_G(h_3,h_6)=1$.
Hence, analogously, we can obtain $d_G(h_3,h^{\prime})=1$ for any $h^{\prime}\in V(H)\setminus\{h_3,h_n,h_{n-1}\}$.
It is easy to see that $G\cong F_{n-1}$ if $\{h_n,h_{n-1}\}\subseteq N_G(h_3)$.
Otherwise, we have
$$G\cong\left\{
\begin{array}{rcl}
F(3;n), &h_n\in N_G(h_3), h_{n-1}\not\in N_G(h_3);\\[6pt]
F(2;n), &h_n\not\in N_G(h_3), h_{n-1}\in N_G(h_3).
\end{array}\right.
$$
While $i\in [k-1]\setminus\{1\}$, similar to the above, we can obtain $G\in \mathcal{F}$.

\medskip\noindent
{\bf Subcase 2.} $|R^*|=2$.\\
In this subcase, we may, without loss of generality, let $h_1$, $h_x\in R^*$ with $3\leq x\leq \lfloor\frac{n}{2}\rfloor+1$. Then $d_H(h_1,h_x)=x-1$.
The proof is simple if $n\leq 9$.
Assume that $n\geq 10$ in the following.

Assume that $x=3$, that is, $\{h_1,h_3\}\subseteq R^*$.
According to Claim 1, we may without loss of generality, let $R=\{h_{3i+2},h_{3i+3}:i\in [k-1]\setminus\{1\}\}\cup\{h_1,h_3\}$ since $|R|=2k$.
Take $\overline{R}=V(H)\setminus R=\{h_{3i+1}:i\in[k]\}\cup\{h_2\}$.
As $h_3$, $h_4$, $h_5$ and $h_6$ are four consecutive vertices of $H$, applying Claim 2 and Lemma \ref{GP-S}, then we have $\{h_3,h_4,h_5,h_6\}\subseteq N_G(\bar{r})$ for some $\bar{r}\in \overline{R}$.
If $\bar{r}=h_2$, then $h_3$ is a 2-vertex with two neighbor vertices $h_2$ and $h_4$ in $G$.
It also implies that $h_5$ is a 3-vertex having three neighbors $h_4$, $_6$ and $h_2$ in $G$.
By Lemmas \ref{GP-C} and \ref{GP-S}, we have $d_G(h_5,h_8)=2$ since $\{h_5,h_6,h_8\}\subseteq R$, which means that $h_8\in N_G(h_2)$.
Then we also have $h_7\in N_G(h_2)$ since $G$ is a maximal outerplane graph.
By similar, we can get $h_9\in N_G(h_2)$.
Repeating the similar discussion above, it is observed that $h^{\prime}\in N_G(h_2)$ for any $h^{\prime}\in V(H)\setminus\{h_2\}$.
Thus we have $G\cong F_{n-1}$.

Next assume that $\bar{r}\in \overline{R}\setminus\{h_2\}$, and without loss of generality let $\bar{r}=h_4$.
Then $d_G(h_4,h_6)=1$, and hence we get that $h_5$ is a 2-vertex of $G$.
Since $\{h_5,h_6,h_8\}\subseteq R$ and $d_G(h_6,h_8)=2$, applying Lemma \ref{GP-C}, we have $d_G(h_5,h_8)=2$ and hence $d_G(h_4,h_8)=1$.
In addition, it follows that $d_G(h_4,h_7)=1$ since $G$ is a maximal outerplane graph.
Similarly, $d_G(h_4,h_9)=1$.
Repeating the similar discussion above,
it is easy to see that $d_G(h_4,h^{\prime})=1$ for any $h^{\prime}\in V(H)\setminus\{h_2,h_4\}$.
Furthermore, we can get $G\cong F_{n-1}$ if $d_G(h_4,h_2)=1$, otherwise $G\cong F(1;n)$.
Analogously, we can obtain that $G\in \{F_{n-1}\}\cup \mathcal{F}$ if $\bar{r}\in \overline{R}\setminus\{h_2,h_4\}$.

Suppose that $5\leq x \leq \lfloor\frac{n}{2}\rfloor+1$.
By Claim 1, we have
\begin{equation*}
\begin{aligned}
|R|& \leq 2\times\big\lfloor\frac{(n-x-2)}{3}\big\rfloor+2+ 2\times\big\lfloor\frac{(x-4)}{3}\big\rfloor+2+2\\
     &=2\times\big\lfloor\frac{(3k-x-1)}{3}\big\rfloor+2\times\big\lfloor\frac{(x-1)}{3}\big\rfloor+4\\
     &=2\times(\big\lfloor\frac{(-x-1)}{3}\big\rfloor+\big\lfloor\frac{(x-1)}{3}\big\rfloor)+2k+4.
\end{aligned}
\end{equation*}
According to our assumption, $gp(G)=2k$ implies that $\big\lfloor\frac{(-x-1)}{3}\big\rfloor+\big\lfloor\frac{(x-1)}{3}\big\rfloor=-2$. It is easy to verify that  $x=3y$ for any $y\in [\lceil\frac{k}{2}\rceil]$. Note that $x\geq 5$. Thus we have $y\geq 2$.

Let $R_1=\{h_{3j},h_{3j+1}:j\in [y-1]\}$ and $R_2=\{h_{3i+2},h_{3i+3}:i\in [k-1]\setminus[y-1]\}$.
Since $|R^*|=2$, we may let $R=R_1\cup R_2\cup\{h_1,h_y\}$.
Set $\overline{R}_1=\{h_{3j-1}:j\in[y]\}$ and $\overline{R}_2=\{h_{3i+1}:i\in [k]\setminus[y-1]\}$. It implies that $V(H)\setminus R=\overline{R}_1\cup \overline{R}_2$.
As $\{h_{n-2},h_{n-1},h_n,h_1\}$ and $\{h_1,h_2,h_3,h_4\}$ are four consecutive vertices sets on $H$, by Claim 2,
$\{h_{n-2},h_{n-1},h_n,h_1\}\subseteq N_G[\bar{r}]$ and $\{h_1,h_2,h_3,h_4\}\subseteq N_G[\bar{r}_1]$ for some $\bar{r}$, $\bar{r}_1\in \overline{R}$.
Analogously to the previous discussion, we can get $G\cong F_{n-1}$ if $\bar{r}=\bar{r}_1$.

Suppose that $\bar{r}\neq\bar{r}_1$, then we claim that $\bar{r}\not\in \overline{R}_1$ in what follows.
If not, it is easy to see that $\bar{r}\neq h_2$. Otherwise, it contradicts $\bar{r}=\bar{r}_1=h_2$ or there are two intersection edges $h_2h_n$ and $h_1\bar{r}_1$ in $G$.
Assume that $\bar{r}=h_{x-1}$ with $x\geq 5$.  By Observation \ref{NV} and Claim 2,
$\{h_x,h_{x-2},h_{x-3}\}\subseteq N_G(h_{x-1})$ as $h_x$, $h_{x-1}$, $h_{x-2}$ and $h_{x-3}$ are four consecutive vertices on $H$.
Analogously, we can get $d_G(h_{x-1},h_\ell)=1$ for any $\ell\in [x-2]$. It leads to $\bar{r}=\bar{r}_1$, which is a contradiction.
Suppose that $\bar{r}\in \overline{R}_1\setminus\{h_2,h_{x-1}\}$, and without loss of generality let $\bar{r}=h_{3j-1}$ with $j\in [y-1]\setminus\{1\}$.
By the above argument, the conclusion holds clearly if $\overline{R}_1=\{h_2,h_{x-1}\}$.
While $|\overline{R}_1\setminus\{h_2,h_{x-1}\}|\neq0$, then $y\geq 3$ and it means that $|R_1|\geq 4$.
Since $h_{3j-3}$, $h_{3j-2}$, $h_{3j-1}$ and $h_{3j}$ are four consecutive vertices of $H$,
applying Claim 2, $\{h_{3j-3},h_{3j-2},h_{3j}\}\subseteq N_G(h_{3j-1})$.
Using the same argument, we can get $\{h_3,h_4\}\subseteq N_G(h_{3j-1})$, that is, $\bar{r}=\bar{r}_1$. It is a contradiction, again.
Then we have $r\not\in \overline{R}_1$. By symmetry, $\bar{r}_1\not\in \overline{R}_2$.

Next assume that $\bar{r}\in \overline{R}_2$ and $\bar{r}_1\in \overline{R}_1$.
If $\bar{r}=h_n$ and $\bar{r}_1=h_2$, it implies that $h_{n-1}$ and $h_3$ are two 2-vertices of $G$.
Similar to the above, we can obtain that $d_G(h_n,h_s)=1$ and $d_G(h_2,h_{s^{\prime}})=1$ for any $s\in[n-1]\setminus [x-1]$ and $s^{\prime}\in [x]\setminus \{2\}$.
The fact that $G$ is a maximal outerplane graph implies that either $d_G(h_1,h_x)=1$ or $d_G(h_n,h_2)=1$.
But if $d_G(h_1,h_x)=1$, then $d_G(h_{n-1},h_3)=4=d_H(h_{n-1},h_3)$, which leads to $h_1\in I_G(h_{n-1},h_3)$ contradicting with our assumption.
Thus we have $d_G(h_n,h_2)=1$, then  $G\cong G_1(1)$.
For the remaining cases, we can apply the similar method as the above repeatedly and obtain that $G\in \mathcal{G^*}$.

\medskip\noindent
{\bf Case 3.} $n=3k+2$.\\
In this case, $gp(G)=2k+1$. We can get $G\cong F_{n-1}$, this proof is similar to that of Case 1, hence omitted here.
\end{proof}

\subsection{Maximal outerplane graphs without internal triangles }
\label{Sub:maxi-out}

In this subsection, we determine the bounds on the gp-numbers for striped maximal outerplane graphs.
For that, we prove one lemma that will be useful in this proof.

Having proved Lemma \ref{D-M}, it would be interesting to know when the lower bound is achieved.
For the sake of this, we introduce the following definition.

\begin{definition}(\cite{Barrett:2019})
For any $n\geq 3$,  if $G_n$ is a connected graph with $V(G_n)=\{v_1,v_2,\ldots,v_n\}$ and $v_iv_j\in E(G_n)$ if and only if $0<|i-j|\leq 2$, then we call $G_n$ the straight linear 2-tree.
\end{definition}

\begin{figure}[htbp]
\centering
\begin{tikzpicture}[scale=1]

	\node[fill=black,circle,inner sep=2pt] at (0,0) {};
	\node[fill=black,circle,inner sep=2pt] at (1.5,0) {};
    \node[fill=black,circle,inner sep=2pt] at (3,0) {};
    \node[fill=black,circle,inner sep=2pt] at (5,0) {};
    \node[fill=black,circle,inner sep=2pt] at (6.5,0) {};
    \node[fill=black,circle,inner sep=2pt] at (8,0) {};
    \node[fill=black,circle,inner sep=2pt] at (0.75,1.5) {};
    \node[fill=black,circle,inner sep=2pt] at (2.25,1.5) {};
    \node[fill=black,circle,inner sep=2pt] at (3.75,1.5) {};
    \node[fill=black,circle,inner sep=2pt] at (5.75,1.5) {};
    \node[fill=black,circle,inner sep=2pt] at (7.25,1.5) {};
    \node [below=0.5mm ] at (0,0) {$v_1$};
    \node [below=0.5mm] at (1.5,0) {$v_3$};
    \node [below=0.5mm] at (3,0) {$v_5$};
    \node [below=0.5mm] at (5,0) {$v_{n-4}$};
    \node [below=0.5mm] at (6.5,0) {$v_{n-2}$};
    \node [below=0.5mm] at (8,0) {$v_n$};
    \node [above=0.5mm] at (0.75,1.5) {$v_2$};
    \node [above=0.5mm] at (2.25,1.5) {$v_4$};
    \node [above=0.5mm] at (3.75,1.5) {$v_6$};
    \node [above=0.5mm] at (5.75,1.5) {$v_{n-3}$};
    \node [above=0.5mm] at (7.25,1.5) {$v_{n-1}$};

	\draw [thick] (0,0)--(1.5,0);
	\draw [thick] (1.5,0)--(3,0);
    \draw [thick][dashed] (3,0)--(5,0);
	\draw [thick](5,0)--(6.5,0);
    \draw [thick](6.5,0)--(8,0);
    \draw [thick](0.75,1.5)--(2.25,1.5);
    \draw [thick](2.25,1.5)--(3.75,1.5);
    \draw [thick][dashed] (3.75,1.5)--(5.75,1.5);
    \draw [thick](5.75,1.5)--(7.25,1.5);
    \draw [thick](0,0)--(0.75,1.5);
    \draw [thick](1.5,0)--(2.25,1.5);
    \draw [thick](3,0)--(3.75,1.5);
    \draw [thick](5,0)--(5.75,1.5);
    \draw [thick](6.5,0)--(7.25,1.5);
    \draw [thick](8,0)--(7.25,1.5);
    \draw [thick](1.5,0)--(0.75,1.5);
    \draw [thick](3,0)--(2.25,1.5);
    \draw [thick](6.5,0)--(5.75,1.5);
\end{tikzpicture}
\caption{Straight linear 2-tree $G_n$. }\label{figure2}
\end{figure}
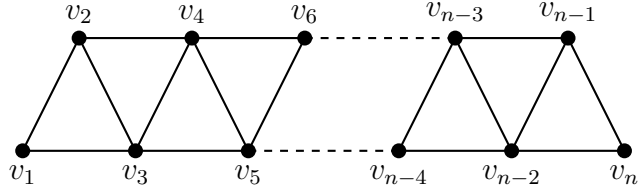

We observe that the straight linear 2-tree $G_n$ is a striped maximal outerplane graph with $\Delta(G_n)=4$. Notice that $G_n$ has exactly two 3-vertices and $n-4$ 4-vertices.

\begin{lemma}\label{MOP-4}
Let $G$ be a maximal outerplane graph of order $n\geq 7$. Then $G\cong G_n$ if and only if $\Delta(G)=4$.
\end{lemma}
\begin{proof}
By the definition of $G_n$, it is clear to see that $\Delta(G)=4$ if $G\cong G_n$. For the converse, suppose that the maximal degree of $G$ is 4.
Let $H$ be a Hamiltonian cycle of $G$ with $V(H)=\{h_1,h_2,\ldots,h_{n}\}$ and $S_{uv}$ be an $u$,$v$-segment on $H$ for any $u$, $v\in V(H)$.
In order to obtain our result, we will prove that two claims in the following.

\medskip\noindent
{\bf Claim 1.} There are exactly two 2-vertices in $G$.\\
Suppose that there exists an internal triangle $T$ in $G$ with vertices $\{h_i,h_j,h_k\}$.
It implies that $\min\{d_H(h_i,h_j),d_H(h_j,h_k),d_H(h_i,h_k)\}\geq 2$, otherwise $T$ is a marginal triangle contradicts  our assumption.
Since $T$ is an internal triangle and $\Delta(G)=4$, $d_G(h_i)=d_G(h_j)=d_G(h_k)=4$.
Assume, without loss of generality, that $h_i$, $h_j$, $h_k$ lie on $H$ with clockwise order.
Since $n\geq 7$, $\max\{d_H(h_i,h_j),d_H(h_j,h_k),d_H(h_i,h_k)\}\geq 3$.
Without loss of generality, let $d_H(h_i,h_j)\geq 3$ and $S_{h_ih_j}=h_i u_1 u_2\ldots u_\ell h_j$ with $\ell\geq 2$.
As $G$ is a maximal outerplane graph, there must be one vertex $u_s$ such that the induce subgraph on $\{h_i,h_j,u_s\}$ is a $K_3$ for $s\in[\ell]$.
It follows that $\max\{d_G(h_i),d_G(h_j)\}\geq 5$, which is a contradiction.
Thus, there exists no internal triangle in $G$.
By Lemma \ref{In-k}, $G$ has exactly two 2-vertices, completing this claim.

\medskip\noindent
{\bf Claim 2.} If $N_G(h_i)=\{x,y,z,w\}$ with $h_i\in I_H(x,y)$, then $d_H(z,w)=1$ for any $i\in [n]$.\\
Assume, without loss of generality, that $x$, $h_i$, $y$, $z$ and $w$ lie on $H$ with clockwise order.
By Lemma \ref{C-M}, $N_G[h_i]$ forms a maximal fan in $G$ with the central vertex $h_i$, then we have $d_G(z,w)=1$.
Note that $d_H(z,w)\geq d_G(z,w)$. Thus $d_H(z,w)\geq 1$.
Indeed $d_H(z,w)=1$. If not, let $S_{zw}=zv_1v_2\ldots v_\ell w$ with $\ell\geq 1$ be a $z$, $w$-segment on $H$. Then $d_G(z)=d_G(w)=4$.
Since the maximal outerplanarity of $G$, there must be one vertex $v_s$ such that the induced subgraph on $z$, $v_s$ and $w$ is a $K_3$ for $s\in [\ell]$.
It means that $\max\{d_G(z),d_G(w)\}\geq 5$, which contradicts $\Delta(G)=4$.
Thus we have $d_H(z,w)=1$, as desired.

Assume that $h_i$, $h_j$ and $h_k$ are any consecutive three vertices of $H$ in clockwise order. Then $N_H(h_j)=\{h_i,h_k\}$.
Next, we will show that $d_G(h_i)=3$ and $d_G(h_k)=4$ if $h_j$ is a 2-vertex of $G$.
If $d_G(h_j)=2$, we have $d_G(h_i,h_k)=1$ and thus $\min\{d_G(h_i),d_G(h_k)\}\geq 3$.
Note that $n\geq 7$. There must be another vertex $h\in V(H)\setminus\{h_i,h_j,h_k\}$ such that the induced subgraph on $\{h,h_i,h_k\}$ is a $K_3$.
Otherwise, there must be a cycle $C_4$ in $G$, which is impossible.
Thus, we can get $\max\{d_G(h_i),d_G(h_k)\}=4$. Furthermore, if $d_G(h_i)=d_G(h_k)=4$, by Lemma \ref{C-M}, there must be one vertex $h\in N_G(h_i)\setminus\{h_j,h_k\}$
such that the induced subgraph on $\{h_i,h_k,h\}$ is a $K_3$. Then $d_H(h_i,h)\geq 2$ and $d_H(h_k,h)\geq 2$.
Thus $\{h_i,h_k,h\}$ forms an internal triangle in $G$, which contradicts Claim 1. Hence, there exactly one 3-vertex in $h_i$ and $h_k$.
Applying Claim 1, it means that $G$ has exactly two pairs of consecutive vertices $u$ and $v$ such that $u$ is a 2-vertex and $v$ is an 3-vertex.

Without loss of generality let $h_i$ be an 3-vertex of $G$ and $N_H(h_i)=\{h_{i^{\prime}}, h_j\}$. Then $N_G(h_i)=\{h_{i^{\prime}}, h_j,h_k\}$.
Based on the above, we will further show that $d_G(h_{i^{\prime}})=4$.
As $n\geq 7$, there must exist one vertex $h^{\prime}\in V(H)\setminus N_G[h_i]$ such that $\{h_{i^{\prime}},h_k,h^{\prime}\}$ forms a $K_3$ in $G$.
Then we have $d_G(h^{\prime})=4$, thus $G$ has $n-4$ 4-vertices.

Let $h_i$ and $h_j$ be two 2-vertices of $G$ with $i<j$. Then we will prove that $d_H(h_i,h_j)=\lfloor\frac{n}{2}\rfloor$ in the following.
Assume that $S_{h_ih_j}=h_iu_1u_2\ldots u_k h_j$ and $S_{h_jh_i}=h_jv_1v_2\ldots v_\ell h_i$ are two segments of $G$ with its clockwise order on $H$ and $\min\{k,\ell\}\geq 1$. Then $k+\ell=n-2$. It also follows that $N_H(h_i)=\{u_1,v_\ell\}$ and $N_H(h_j)=\{u_k,v_1\}$.
If $d_H(h_i,h_j)<\lfloor\frac{n}{2}\rfloor$, then $|k-\ell|\geq 2$.
Suppose, without loss of generality, that $\ell> k$, i.e., $\ell=k+2$.
Since $h_i$ and $h_j$ are two 2-vertices of $G$, then we have $d_G(u_1,v_\ell)=1$ and $d_G(u_k,v_1)=1$.
By the above argument, we know that both $h_i$ and $h_j$ have a neighbor vertex of degree 3 in $G$.
Without loss of generality, let $d_G(u_1)=3$.
Then $d_G(v_\ell)=4$. It implies that there must be another vertex $h\in N_G(u_\ell)$ except $h_i$, $u_1$ and $v_{\ell-1}$,
by Claim 2, we can obtain $d_H(u_1,h)=1$. It leads to $h=u_2$, that is, $d_G(v_\ell,u_2)=1$.
Similarly, $d_G(v_{\ell-1},u_2)=1$ if $d_G(u_2)=4$.
Recall that $\ell>k$ and $G$ has $n-4$ 4-vertices.
Analogously, $d_G(u_s,v_{\ell-s+1})=1$ and $d_G(u_s,v_{\ell-s+2})=1$ for any $s\in[k]\setminus\{1\}$.
Since $d_G(u_k,v_1)=1$ and $\ell-k+1\geq 2$, it follows that $d_G(u_k)\geq 5$, which is impossible.
Similarly, we can get the same contradiction if $d_G(u_1)=4$.
In consequence, we can conclude that $d_H(h_i,h_j)=\lfloor\frac{n}{2}\rfloor$.

By the above Claims and arguments, we can obtain that $G\cong G_n$ if $\Delta(G)=4$, as desired.
\end{proof}

With the help of Lemma \ref{MOP-4}, we can present an extremal graph for achieving the lower bound in Lemma \ref{D-M}.
According to Theorem \ref{F^*-I-N}, we determine an upper bound on the gp-number of $G$ without internal triangle and characterize its corresponding extremal graphs in the following.

\begin{theorem}\label{SM-G}
Let $G$ be a striped maximal outerplane graph of order $n\geq 5$. Then we have
$$3\leq gp(G)\leq \lfloor\frac{2n}{3}\rfloor$$
with left equality iff $G\cong G_n$ and right equality iff $G\in\{F_{n-1}, F(1;n),F(n-2;n),\\G_1(1),G_2(t)\}$.
\end{theorem}
\begin{proof}
Let $R$ be a gp-set of $G$.
By Theorem \ref{F^*-I-N}, the upper bound holds clearly. Furthermore, since $G$ has no internal triangles, then we have $G\in \{F_{n-1}, F(1;n),\\F(n-2;n),G_1(1),G_2(t)\}$.

We observe that $|R|\geq 3$. Next, it needs to prove that the left equality holds.
Assume first that $G$ is a maximal outerplane graph satisfying $gp(G)=3$. Recall that $G$ is a 2-connected graph.
Applying  Lemma \ref{D-M}, we can get $gp(G)\geq 4$ if $\Delta(G)\geq 5$. Thus, based on our assumption, we have $2\leq d_G(v)\leq 4$ for any vertex $v\in V(G)$.
But if $\Delta(G)\leq 3$, then $G$ is not a maximal outerplane graph contradicting the maximal outerplanarity of $G$.
Then $\Delta(G)=4$ if $gp(G)=3$. By Lemma \ref{MOP-4}, it is easy to see that $G\cong G_n$.

Suppose now that $G\cong G_n$.
Then we will show that $|R|\leq 3$ in what follows. Let $H$ be a Hamiltonian cycle of $G$ and $V(H)=\{h_0,h_1,\ldots,h_{n-1}\}$ with clockwise order on $H$.
Conversely, suppose that $|R|\geq 4$, which means that $|V(H)\cap R|\geq 4$.
For any positive integer $d\geq 2$, if $n=2d$, we may assume without loss of generality, that $h_0$, $h_d$ are two 2-vertices and $h_1$, $h_{d+1}$ are two 3-vertices of $G$, respectively.
By the structure of $G$, $d_G(h_i,h_{n-i+1})=1$ and $d_G(h_i,h_{n-i})=1$ for any $i\in[d-1]\setminus\{1\}$.
Let $S_{h_0h_d}=h_0h_1\ldots h_d$ and $S_{h_dh_0}=h_d h_{d+1}\ldots h_n h_0$ be two segments on $H$, denoted by $h_0$, $h_d$-segment and $h_d$, $h_0$-segment, respectively.
It is clear to see that $|V(S_{h_0h_d})\cap R|\leq 2$. Otherwise there must be three vertices in $V(S_{h_0h_d})\cap R$ lying on the same geodesic in $G$, which is impossible.
Analogously, $|V(S_{h_dh_0})\cap R|\leq 2$.
Based on our assumption, we have $|R|=4$ and thus $|V(S_{h_0h_d})\cap R|=2=|V(S_{h_dh_0})\cap R|$.

Assume, without loss of generality, that $V(S_{h_0h_d})\cap R=\{x,y\}$ and $V(S_{h_dh_0})\cap R=\{z,w\}$ with clockwise order on $H$.
Since $H$ is a cycle, there must be three vertices in $R\cap V(H)$ such that they lying on the same geodesic in $H$.
Without loss of generality, let $y\in I_H(x,z)$.
By the structure of $G$, it is obvious that $d_G(x,z)=d_G(x,y)+d_G(y,z)$. Then we have $y\in I_G(x,z)$, it contradicts $\{x,y,z\}\subseteq R$.
The similar contradiction can be obtained if $z\in I_G(y,w)$ or $w\in I_G(x,z)$.
Thus we have $|R|\leq 3$.

Using the similar arguments as above, we can obtain the same conclusion if $n$ is odd. Hence, $|R|=3$ if $G\cong G_n$.
We complete this proof of the theorem.
\end{proof}

\subsection{Maximal outerplane graphs with internal triangles }
\label{Sub:maxi-out-internal}

In this subsection, we concentrate on the bounds on the gp-numbers of $G$ with internal triangles. We will give a necessary condition for attaining the lower bound. And we also characterize some extremal graphs when the upper bound is achieved.

It is easy to see from Lemma \ref{In-k} that  the number of 2-vertices is related to the number of internal triangles in $G$.
Hence, the characterization of the bound on the number of 2-vertices in $G$ is very important in our remaining proof.
First we give some useful definitions and lemmas as preparation in what follows. Let $m\geq 3$ be a positive integer.

\begin{definition}(\cite{Gallian:2014})
The sunflower graph $SF_{2m+1}$ is a graph obtained by taking a wheel with the central vertex $v$ and the $m$-cycle $v_0$, $v_1$,\ldots,$v_{m-1}$ combined with additional vertices $u_0$, $u_1$,\ldots,$u_{m-1}$, where $u_i$ is joined by edges to $v_i$, $v_{i+1}$ and $i+1$ is taken from modulo $m$.
\end{definition}
The sunflower graph $SF_9$ can be depicted as in the left graph of Figure 3.

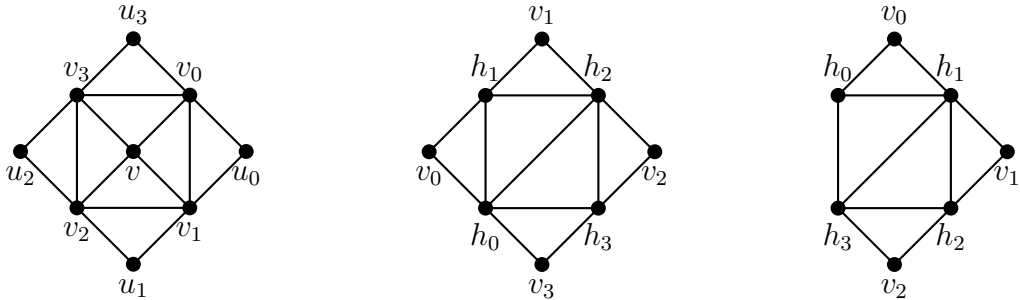
\begin{figure}[htbp]
\centering
\begin{tikzpicture}[scale=1.5]

	\node[fill=black,circle,inner sep=2pt] at (0,0) {};
	\node[fill=black,circle,inner sep=2pt] at (1,0) {};
    \node[fill=black,circle,inner sep=2pt] at (-1,0) {};
    \node[fill=black,circle,inner sep=2pt] at (-0.5,-0.5) {};
    \node[fill=black,circle,inner sep=2pt] at (0,-1) {};
    \node[fill=black,circle,inner sep=2pt] at (0.5,-0.5) {};
    \node[fill=black,circle,inner sep=2pt] at (-0.5,0.5) {};
    \node[fill=black,circle,inner sep=2pt] at (0.5,0.5) {};
    \node[fill=black,circle,inner sep=2pt] at (0,1) {};

    \node [below=0.5mm ] at (0,0) {$v$};
    \node [below=0.5mm ] at (1,0) {$u_0$};
    \node [below=0.5mm ] at (-1,0) {$u_2$};
    \node [below=0.5mm ] at (-0.5,-0.5) {$v_2$};
    \node [below=0.5mm ] at (0.5,-0.5) {$v_1$};
    \node [above=0.5mm ] at (0.5,0.5) {$v_0$};
    \node [above=0.5mm ] at (-0.5,0.5) {$v_3$};
    \node [below=0.5mm ] at (0,-1) {$u_1$};
    \node [above=0.5mm ] at (0,1) {$u_3$};

	\draw [thick] (0,0)--(0.5,-0.5);
    \draw [thick] (0,0)--(-0.5,-0.5);
    \draw [thick] (0,0)--(0.5,0.5);
    \draw [thick] (0,0)--(-0.5,0.5);
    \draw [thick] (0,-1)--(0.5,-0.5);
    \draw [thick] (0,-1)--(-0.5,-0.5);
    \draw [thick] (-0.5,-0.5)--(0.5,-0.5);
    \draw [thick] (0.5,-0.5)--(0.5,0.5);
    \draw [thick] (1,0)--(0.5,-0.5);
    \draw [thick] (1,0)--(0.5,0.5);
    \draw [thick] (0.5,0.5)--(-0.5,0.5);
    \draw [thick] (0.5,0.5)--(0,1);
    \draw [thick] (-0.5,0.5)--(0,1);
    \draw [thick] (-0.5,0.5)--(-0.5,-0.5);
    \draw [thick] (-0.5,0.5)--(-1,0);
    \draw [thick] (-0.5,-0.5)--(-1,0);
\end{tikzpicture}
\hspace{15mm}
\begin{tikzpicture}[scale=1.5]

	\node[fill=black,circle,inner sep=2pt] at (0,0) {};
	\node[fill=black,circle,inner sep=2pt] at (1,0) {};
    \node[fill=black,circle,inner sep=2pt] at (1,1) {};
    \node[fill=black,circle,inner sep=2pt] at (0,1) {};
    \node[fill=black,circle,inner sep=2pt] at (0.5,-0.5) {};
    \node[fill=black,circle,inner sep=2pt] at (-0.5,0.5) {};
    \node[fill=black,circle,inner sep=2pt] at (0.5,1.5) {};
    \node[fill=black,circle,inner sep=2pt] at (1.5,0.5) {};

     \node [below=0.5mm ] at (0,0) {$h_0$};
     \node [below=0.5mm ] at (1,0) {$h_3$};
     \node [above=0.5mm ] at (0,1) {$h_1$};
     \node [above=0.5mm ] at (1,1) {$h_2$};
     \node [below=0.5mm ] at (0.5,-0.5){$v_3$};
     \node [below=0.5mm ] at (-0.5,0.5){$v_0$};
     \node [above=0.5mm ] at (0.5,1.5){$v_1$};
     \node [below=0.5mm ] at (1.5,0.5){$v_2$};

	\draw [thick] (0,0)--(0.5,-0.5);
    \draw [thick] (0,0)--(-0.5,0.5);
    \draw [thick] (0,0)--(0,1);
    \draw [thick] (0,0)--(1,1);
    \draw [thick] (0,0)--(1,0);
    \draw [thick] (-0.5,0.5)--(0,1);
    \draw [thick] (0,1)--(0.5,1.5);
    \draw [thick] (0,1)--(1,1);
    \draw [thick] (1,1)--(0.5,1.5);
    \draw [thick] (1,1)--(1.5,0.5);
    \draw [thick] (1,1)--(1,0);
    \draw [thick] (1,0)--(1.5,0.5);
    \draw [thick] (1,0)--(0.5,-0.5);
\end{tikzpicture}
\hspace{15mm}
\begin{tikzpicture}[scale=1.5]

	\node[fill=black,circle,inner sep=2pt] at (0,0) {};
	\node[fill=black,circle,inner sep=2pt] at (1,0) {};
    \node[fill=black,circle,inner sep=2pt] at (1,1) {};
    \node[fill=black,circle,inner sep=2pt] at (0,1) {};
    \node[fill=black,circle,inner sep=2pt] at (0.5,-0.5) {};
    \node[fill=black,circle,inner sep=2pt] at (0.5,1.5) {};
    \node[fill=black,circle,inner sep=2pt] at (1.5,0.5) {};

    \node [below=0.5mm ] at (0,0) {$h_3$};
     \node [below=0.5mm ] at (1,0) {$h_2$};
     \node [above=0.5mm ] at (0,1) {$h_0$};
     \node [above=0.5mm ] at (1,1) {$h_1$};
     \node [below=0.5mm ] at (0.5,-0.5){$v_2$};
     \node [above=0.5mm ] at (0.5,1.5){$v_0$};
     \node [below=0.5mm ] at (1.5,0.5){$v_1$};

	\draw [thick] (0,0)--(0.5,-0.5);
    \draw [thick] (0,0)--(0,1);
    \draw [thick] (0,0)--(1,1);
    \draw [thick] (0,0)--(1,0);
    \draw [thick] (0,1)--(0.5,1.5);
    \draw [thick] (0,1)--(1,1);
    \draw [thick] (1,1)--(0.5,1.5);
    \draw [thick] (1,1)--(1.5,0.5);
    \draw [thick] (1,1)--(1,0);
    \draw [thick] (1,0)--(1.5,0.5);
    \draw [thick] (1,0)--(0.5,-0.5);
  \end{tikzpicture}
\caption{ Sunflower graph $SF_9$ and generalized sunflower graphs $GSF_8$ and $GSF_7$. }\label{figure3}
\end{figure}

Based on the definition of sunflower graph proposed by Gallian in \cite{Gallian:2014}, we give the following definition of generalized sunflower graph.

\begin{definition}\label{G-SF}
Let $H$ be a Hamiltonian cycle of maximal outerplane graph $G$ with naturally adjacent vertices $h_0$, $h_1$,\ldots,$h_{m-1}$.
The generalized sunflower graph is a graph obtained by taking $H$ of $G$ combined with additional vertices $v_0$, $v_1$,\ldots,$v_{x-1}$,
where $v_i$ is joined by edges to $h_i$, $h_{i+1}$, where $i+1$ is taken from modulo $m$ and $x\in \{m-1,m-2\}$.
\end{definition}

In this paper, we denote the generalized sunflower graph of order $n$ by $GSF_n$.
According to the Definition \ref{G-SF}, it is easy to see that $n\in\{2m-1,2m\}$.
Moreover, the generalized sunflower graph is a maximal outerplane graph with internal triangles if $m\geq 4$.
As an example, $GSF_8$ and $GSF_7$ can be depicted as in the middle and right of Figure 3, respectively.

Next, we prove that the upper bound on the total number of 2-vertices in $G$ with internal triangle, and characterize its corresponding extremal graphs when the upper bound is achieved.

\begin{lemma}\label{I-U}
Let $G$ be a maximal outerplane graph with order $n$ and $k\geq 1$ internal triangles. Then $k\leq \lfloor\frac{n}{2}\rfloor-2$ with equality holding if and only if $G\cong GSF_n$.
\end{lemma}
\begin{proof}
By Lemma \ref{In-k}, there must be $k+2$ 2-vertices in $G$ since $G$ has $k$ internal triangles, and thus $G$ has $k+2$ marginal triangles.
Applying Lemma \ref{Diam}, we have $n-2-k\geq k+2$, that is, $k\leq \frac{n-4}{2}$. It is obvious that $k\leq \lfloor\frac{n}{2}\rfloor-2$ since $k$ is a positive integer.

Let $H$ be a Hamiltonian cycle of $G$ and $V(H)=\{h_1,h_2,\ldots,h_{n}\}$ with natural adjacencies.  Assume that $D_2$ is the set of 2-vertices in $G$ and $V_2$ is the set of two consecutive vertices of $H$. Note that $|V_2\cap D_2|\leq 1$ for any $V_2\subseteq V(H)$.
By the structure of $GSF_n$, it is easy to see that $GSF_n$ has $\lfloor\frac{n}{2}\rfloor-2$ internal triangles.
Next, we may assume that $G$ has $k$ internal triangles with $k=\lfloor\frac{n}{2}\rfloor-2$.
We will consider the following two cases according to the parity of $n$.

\medskip\noindent
{\bf Case 1.} $n$ is even.\\
In this case, $k=\frac{n}{2}-2$. Applying Lemma \ref{In-k}, it is observed that $|D_2|=k+2=\frac{n}{2}$.
According to the property of $D_2$, it is easy to see that $d_H(h,h^{\prime})\geq 2$ for any two vertices $h$, $h^{\prime}\in D_2$.
Then we have $|V_2\cap D_2|=1$ since $|D_2|=\frac{n}{2}$.
So we may assume, without loss of generality, that $D_2=\{h_{2i}:i\in[k+2]\}$. Take $\overline{D}_2=V(H)\setminus D_2=\{h_{2i-1}:i\in[k+2]\}$.

Let $u$, $v$ and $w$ be any three consecutive vertices of $H$  with clockwise order.
It is obvious that $d_G(u,w)=1$ if $v\in D_2$.
Similarly, we can get $d_H(h_{2i-1},h_{2i+1})=1$ and $d_H(h_{n-1},h_1)=1$ for any $i\in[k+1]$. It follows that there exits a cycle $C:=h_1h_3\ldots h_{n-1}h_1$ in $G$.
Since $G$ is a maximal outerplane graph, then the induced subgraph on $\overline{D}_2$ is also a maximal outerplane graph and $C$ is a Hamiltonian cycle of it.
Furthermore, the induced subgraph on $\{h_{2i},h_{2i-1},h_{2i+1}\}$ is a $K_3$ for any $i\in[k+1]$, and $h_{n-1}$, $h_n$ and $h_1$ three vertices of $G$ also induces a $K_3$.
Therefore, we have $G\cong GSF_n$.

\medskip\noindent
{\bf Case 2.} $n$ is odd.\\
In this case, we can get $k=\frac{n-5}{2}$. It follows from Lemma \ref{Diam} that $n-2-k=k+3$, which means that $G$ has $k+3$ marginal triangles.
According to Lemma \ref{In-k}, there must be $k+2$ marginal triangles containing 2-vertex in $G$.
Then we will claim that there exactly one $V_2$ such that $V_2\cap D_2=\emptyset$, where $V_2\subseteq V(H)$.

Conversely, suppose without loss of generality, that there exist two pairs adjacent of vertices $h_i$, $h_{i+1}$ and $h_j$, $h_{j+1}$ such that $\{h_i,h_{i+1},h_j,h_{j+1}\}\cap D_2=\emptyset$, where $1\leq i<j\leq n-1$.
Then $|\{h_i,h_{i+1}\}\cap\{h_j,h_{j+1}\}|\leq 1$. Recall that $|V_2\cap D_2|\leq 1$ for any $V_2\subseteq V(H)$.
If $i+1=j$, that is, $|\{h_i,h_{i+1}\}\cap\{h_j,h_{j+1}\}|=1$, then we have
$$|D_2|\leq \lceil\frac{n-3}{2}\rceil=\frac{n-3}{2}=k+1.$$
It is a contradiction.
While $i+1<j$, let $S_{h_{i+1} h_j}$, $S_{h_{j+1} h_i}$ be the $h_{i+1}$, $h_j$-segment and $h_{j+1}$, $h_i$-segment on $H$, respectively.
Since $\{h_i,h_{i+1},h_j,h_{j+1}\}\cap D_2=\emptyset$,
then we have
$$|D_2|\leq \lceil\frac{j-i-1}{2}\rceil+\lceil\frac{n-j+i-3}{2}\rceil\leq \frac{n-2}{2}<k+2.$$
It is a contradiction, again.
As a consequence, there is at most one $V_2\subseteq V(H)$ such that $V_2\cap D_2=\emptyset$.
In fact, there must be exactly one $V_2\subseteq V(H)$ such that $V_2\cap D_2=\emptyset$ since $n$ is odd, completing this claim.

By the above claim, we may assume without loss of generality, that $D_2=\{h_{2j+1}:j\in[k+2]\}$. Take $\overline{D}_2=\{h_{2j}:j\in[k+2]\}\cup\{h_1\}$.
Similar to the above Case 1, thus we have $G\cong GSF_n$.
\end{proof}

Theorem \ref{SM-G} gives a necessary and sufficient condition when the bounds on the gp-numbers for a striped maximal outerplane graph are achieved.
Then we focus on the bounds on the gp-number for a maximal outerplane graph $G$ with internal triangles.
By the structure of $GSF_7$ (see the right graph of Figure 3), it is easy to verify that $gp(GSF_7)=4$.

\begin{theorem}
For any positive integer $k\geq 1$, let $G$ be a maximal outerplane graph with $k$ internal triangles. Then
$$k+2\leq gp(G)\leq \lfloor\frac{2n}{3}\rfloor.$$
with left equality  if $G\cong GSF_n(n\geq 8)$.
Moreover, the right equality holding if and only if $G\in\big(\mathcal{F}\setminus\{F(1;n),F(n-2;n)\}\big)\cup \big(\mathcal{G}^*\setminus\{G_1(1),G_2(t)\}\big)$.
\end{theorem}
\begin{proof}
Let $H$ be a Hamiltonian cycle of $G$ with $V(H)=\{h_1,h_2,\ldots,h_n\}$ and $D_2$ be the set of 2-vertices of $G$. By Lemma \ref{In-k}, we can get $|D_2|=k+2$.
It is easy to verify that $D_2$ is a general position set of $G$, thus we have $gp(G)\geq k+2$.
And in view of Theorem \ref{F^*-I-N}, the upper bound holds clearly and its corresponding extremal graphs are characterized.

Assume that $G\cong GSF_n$ satisfies $n\geq 8$ in the remaining proof. Then we will prove that $gp(G)=k+2$ in what follows. Let $R$ be a gp-set of $G$.
By the above argument, we obtain $|R|\geq k+2$. Next, it suffices to prove that $|R|\leq k+2$. The proof is simple if $8\leq n\leq 15$, so we may let $n\geq 16$ in the following.

Conversely, assume that $|R|\geq k+3\geq 4$.
Then it implies that there must be at least one pair of adjacent vertices of $H$ belonging to $R$.
Suppose that $R$ contains $m$ pairs of adjacent vertices on $H$, and let $(x_i,y_i)$ be the $i$-th pair of adjacent vertices of $H$ in $R$ with clockwise order.
Let $a_1$, $a_2$, $a_3$ and $a_4$ be the four consecutive vertices of $H$. Define $V_3^{c}(a_4)=\{a_1, a_2, a_3\}$ and $V_3(a_1)=\{a_2, a_3, a_4\}$ on $H$, respectively.
For convenience, set $\overline{D}_2=V(H)\setminus D_2$. Based on the parity of $n$, we divide into the following two cases to prove this result.

\medskip\noindent
{\bf  Case 1.} $n$ is even.\\
In this case, it follows from Lemmas \ref{In-k} and \ref{I-U} that $k+2=\frac{n}{2}$.
So we may assume, without loss of generality, that $D_2=\{h_{2j}:j\in[k+2]\}$.
Take $\overline{D}_2=\{h_{2\ell-1}:\ell\in[k+2]\}$.
The following claim will be useful to our main proof.

\medskip\noindent
{\bf Claim 1.} For any $i\in[m]$, if $x_i\in \overline{D}_2$, then $V_3^{c}(x_i)\cap R=\emptyset$.\\
Suppose without loss of generality that $x_i=h_1$, thus we have $V_3^{c}(h_1)=\{h_n,h_{n-1},h_{n-2}\}$ and $h_{n-1}\in \overline{D}_2$.
It is obvious that $h_n\not\in R$ since $d_G(h_n,h_1)=1$ and $d_G(h_n,h_2)=2$.
By the definition of $GSF_n$, it implies that $d_G(h_{n-1},h_1)=1$ as $\{h_{n-1}, h_1\}\subseteq\overline{D}_2$.
Since $N_G(h_2)=\{h_1,h_3\}$, we have $d_G(h_2,h_{n-1})=2$.
Hence $h_{n-1}\not\in R$. Similarly, $h_{n-2}\not\in R$. We complete this claim.

Analogously, $V_3(y_i)\cap R=\emptyset$ if $y_i\in \overline{D}_2$ for any $i\in[m]$.
Let $M(G)$ be the set of potential vertices of $R$ in $G$, i.e., the vertex $u\in M(G)$ is probably in $R$.
Take $\overline{M}(G)=V(G)\setminus M(G)$. Then we observe that $v\not\in R$ if $v\in\overline{M}(G)$.
Next, we will prove that $|M(G)|\leq n-3m$.

Without loss of generality, let $x_1$, $x_2$,\ldots,$x_m$ lie on $H$ with clockwise order.
By Claim 1,  it is obvious that $|M(G)|<n-3$ if $m=1$.
Then assume that $m\geq 2$.
For any $i\in[m]$, if $x_i\in \overline{D}_2$, it implies that $y_i\in D_2$, and we have $V_3^{c}(x_i)\cap V_3^{c}(x_{i+1})=\emptyset$ since Claim 1.
Thus $|M(G)|< n-3m$.

While there exists one $i\in [m-1]$ such that $x_i\in\overline{D}_2$ and $x_{i+1}\in D_2$, then we will prove that $|V_3^{c}(x_i)\cap V_3(y_{i+1})|\leq 1$.
Assume, without loss of generality, that $x_i=h_1$, thus we have $V_3^{c}(x_i)=\{h_n,h_{n-1},h_{n-2}\}$.
If $x_{i+1}\in D_2\setminus\{h_2,h_n,h_{n-2},h_{n-4},h_{n-6}\}$, it is easy to verify that $|V_3^{c}(x_i)\cap V_3(y_{i+1})|=0$.
And if $x_{i+1}=h_{n-4}$, then $y_{i+1}=h_{n-3}$ and $h_{n-3}\in R$.
Since $h_{n-3}\in \overline{D}_2$, we have $d_G(h_{n-3},h_1)\leq 2$.
 Note that $h_1\in I_G(h_2,h_{n-3})$ if $d_G(h_1,h_{n-3})=1$.
Then $d_G(h_{n-3},h_1)=2$. It implies that $d_G(h_{n-4},h_1)=2$ and $d_G(h_{n-4},h_2)=2$. And since $h_{n-4}$, $h_2$ are two 2-vertices of $H$ and $N_G(h_{n-4})\cap N_G(h_2)=\emptyset$, which leads to $d_G(h_{n-4},h_2)\geq 3$. It is impossible.
While $x_{i+1}=h_{n-6}$, then $y_{i+1}=h_{n-5}$ and $|V_3(y_{i+1})\cap V_3^{c}(x_i)|=1$.
By using Lemma \ref{GP-S}, it implies that $h_{n-7}\not\in R$. Then there must be at least three vertices of $H$ in $\overline{M}(G)$ if $x_i$, $y_i\in R$, for some $i\in [m]$.
It implies that $|M(G)|\leq n-3m$, completing this claim.

Suppose that these potential vertices can be divided into $m$ parts in $H$ and $n_i$ is the number of vertices for each part.
Then we have
\begin{equation*}
\begin{aligned}
|R|& \leq \lceil\frac{n_1}{2}\rceil+\lceil\frac{n_2}{2}\rceil+\cdots+\lceil\frac{n_m}{2}\rceil+2m\\
     &\leq \frac{n_1+1}{2}+\frac{n_2+1}{2}+\cdots+\frac{n_m+1}{2}+2m\\
     &= \frac{n-5m+m}{2}+2m\\
     &= k+2.
\end{aligned}
\end{equation*}
It contradicts with our assumption. Hence $|R|\leq k+2$, then we have $gp(G)=2k$.

\medskip\noindent
{\bf Case 2.} $n$ is odd.\\
In this case, we know that $|D_2|=k+2$. Without loss of generality, let $D_2=\{h_{2i+1}:i\in[k+2]\}$. Take $\overline{D}_2=\{h_{2i}:i\in[k+2]\}\cup\{h_1\}$.
Similar to the previous discussion of Case 1, we can get $|R|\leq k+2$, hence omitted here. Thus we have $gp(G)=k+2$ if $G\cong GSF_n$.
\end{proof}

\end{document}